\documentclass[11pt, reqno]{amsart}
\usepackage{amsmath, amsthm, amscd, amsfonts, amssymb, graphicx, color}
\usepackage[colorlinks=true, linkcolor=blue,urlcolor=blue]{hyperref}
\input xy \xyoption{all}
\usepackage[all]{xy}
\usepackage{layout}
\usepackage{tikz}

\newtheorem{theorem}{Theorem}[]

\newtheorem{proposition}[theorem]{Proposition}

\theoremstyle{definition}

\newtheorem{example}[theorem]{Example}

\theoremstyle{remark}

\numberwithin{equation}{section}

\newcommand{\s}{\mbox{$\ast$}}

\newcommand{\g}{generalized  }
\newcommand{\rr}{Rickart  }

\begin{document}
\setcounter{page}{1}

\title[artinian rings which are not generalized Rickart]{\textbf{artinian rings which are not generalized Rickart}}

\author[ Ali Shahidikia]{ Ali Shahidikia }
\address{Department of Mathematics, Dezful branch, Islamic Azad University, Dezful, Iran}
\email{\textcolor[rgb]{0.00,0.00,0.84}{ali.Shahidikia@iaud.ac.ir}}


\begin{abstract}

In this note, we show that there exist  non-unital right artinian rings  which
are not  \g  Rickart. In particular, we
provide  examples to show that, \cite[Corollary 2.31]{ung} is not true for non-unital artinian rings.
\end{abstract}

\maketitle
Throughout this note $R $ denotes an associative ring without unity and $\ast$ is used to indicate an involution on a ring.  An idempotent  element $x\in R$ is called a \textit{projection} if $x^{\ast} =x$. Rickart  \cite{18} in 1946 studied C*-algebras  which satisfy the  condition that the right annihilator of every single element is generated by a projection. Rickart also showed that all von Neumann algebras satisfy this property. These algebras were later named \textit{Rickart C*-algebras} by Kaplansky.\par

Kaplansky in  1950 showed that von Neumann algebras satisfy a stronger annihilator condition, namely, that these are rings
 with identity in which the right annihilator of any nonempty subset is generated by an idempotent. He termed a ring with this property
  a Baer ring to honor  R.~Baer who had studied this condition in 1940. Also, a \s-ring $ R $ is called a Baer  $ \ast $-ring if the
   right annihilator of every nonempty subset  is generated by a projection  as a right ideal. Kaplansky recognized that the notions of a Baer ring and a Baer \s-ring provide a framework to study the algebraic properties
     of operator algebras and each is interesting in its own right. The theory of Baer rings, Baer \s-rings, and AW*-algebras
      (C*-algebras which are Baer \s-rings) is studied in \cite{berb} and \cite{Kap65}.\par

Maeda in 1960  defined a Rickart ring. He called a ring \textit{right (left) Rickart} if the right (left) annihilator of any single element is generated by an idempotent. It is clear that every Baer ring is right and left Rickart. The same year, Hattori  introduced the notion of a \textit{right p.p.}-ring, namely a ring in which every principal right ideal is projective. It was later discovered that a right Rickart ring is precisely
the same as a right p.p.-ring. Also, Berberian in \cite{berb} defined a Rickart \s-ring, if the right  annihilator of any single element is generated by a projection.\par

Recall from \cite{oho}, that a ring $ R $ is  \textit{generalized right}  \textit{principally projective} (\textit{generalized right p.p.} for short) if for any $x\in R$, the right ideal $x^nR$ is projective for some positive integer $  n$, depending on $  x$, or equivalently, if for any element $ x\in  R$, the right annihilator $r_R(x^n)$ is generated by an idempotent for some positive integer $  n$. Left cases may be defined analogously. Generalized p.p.-rings (which are also called generalized Rickart rings) have been studied in \cite{oho},\cite{oho2}, \cite{y} and \cite{x}.\par

In \cite[Theorem 2.30]{ung}, Ungor  et al. show  that  every
finitely generated module over a right artinian ring is $ \pi
$-\textit{Rickart}. They deduce that every right artinian ring is generalized right p.p., see \cite[Corollary 2.31]{ung}.

Now we show that there exist non-unital right artinian rings, which are not  \g
Rickart. In particular, this shows that, \cite[Corollary 2.31]{ung},
is not true for non-unital  artinian rings.

We do not know any example of a unital artinian ring which is not generalized p.p.

\begin{theorem}\label{conter}
 Let $ R $ be a non-unital  right  artinian   ring  with the following properties:
\begin{itemize}
\item[$(i)$] For $ a\in R $ and each positive  integer $ n $, $ 2^na=0 $ implies that $ a=0 $;
\item[$(ii)$] The equation $ 2x^2-x=0 $ has only the trivial solution  in $ R $;
\end{itemize}
and let $S = RG$ be the group ring of the group $G =\{e,g\}$ of
order $ 2 $ over the ring $ R $. Then $ S $ is a right artinian
ring which is not \g right p.p.
\end{theorem}
 \begin{proof}
By \cite[p.217, Exercise 2]{intro}, $ S$ is right artinian. But, we show that $S $ is not a \g
right p.p.-ring and hence it is not  a \g \rr\s-ring. Note that for each positive integer $ n
$, $ (e+g)^n=2^{n-1}e+2^{n-1}g $, $ (e-g)\in r_S(e+g)^n $, so $ r_S(e+g)^n\not=0 $. If $ ae+bg\in S $ is a nontrivial idempotent such
that $ ae+bg\in r_S(e+g)^n $, then we must have $ 2^{n-1}a+2^{n-1}b=0
$, $ ab+ba=b $ and $ a^2+b^2=a $. So $ b=-a  $, then $
2a^2=a$ which contradicts the assumption that the equation $
2x^2-x=0 $ has no nontrivial solutions in $ R $. Thus  $ r_S(e+g)^n
$ can not be generated by an idempotent of $ S $.
 \end{proof}
 \begin{proposition}\label{T(R)1}
 If a ring $R$ satisfies the assumptions of Theorem \ref{conter}, then the triangular matrix ring $T_n(R)$   also satisfies these properties.
 \end{proposition}
 \begin{proof}
It is clear that  $T_n(R)$ satisfies the assumption $(i)$ of Theorem \ref{conter}. Let \\
$ M=(a_{ij})\in T_n(R) $, where $ a_{ij}=0 $ for $i> j$. Suppose 
$ 2M^2=M $. Then for $ 1\leq i\leq n $, $ 2a_{ii}^2-a_{ii}=0 $, so $ a_{ii}=0 $. Now $ 2M^2=M $ implies that $ a_{12}=a_{23}=\cdots=a_{n-1,n}=0 $. By continuing this process, we conclude that $ a_{ij}=0 $, for all $1\leq i \leq j\leq n$, so $M=0$ and   $T_n(R)$ satisfies the assumption $(ii)$ of Theorem \ref{conter}.
 \end{proof}
 
 We now provide some examples which satisfy the assumptions of Theorem \ref{conter}, see \cite{Fine} for more details.
 \begin{example}\label{ex}
 Let $ R $ be  one of the following finite rings of order $ p^2 $:
\begin{itemize}
\item[$(1)$] $A=\langle a \mid p^2a=0,a^2=pa\rangle$;
\item[$(2)$]$B=\langle a \mid p^2a=0,a^2=0\rangle$;
\item[$(3)$]$C=\langle a,b \mid pa=pb=0,a^2=b, ab=0\rangle$;
\item[$(4)$] $D=\langle a,b \mid pa=pb=0,a^2=b^2=0\rangle$.
\end{itemize}
where $p$ is a prime number $\not= 2$.  Now, let $S = RG$ be the
group ring of the group $G =\{e,g\}$ of order $ 2 $ over the ring $
R $. Since $ S$ is a finite ring, $ S $ is artinian. Also, since characteristic of $ R $ is $ p^2 $ and $ (p^2,2^n)=1 $,
for each $ n $, $ R $ satisfies Condition $ (i) $ of Theorem
\ref{conter}. \par
If $ R=A $ and   $x\in R$, then
$$ 2x^2=
\begin{cases}
0  &  \textrm{ if } x=kp, \textrm{ where }k=1,\cdots,p-1\\
2px  &  \textrm{ otherwise.}
\end{cases}
$$\par
If $ R=B $ and   $x\in R$, then $ 2x^2=0$.
\par If $ R=C $ and   $x\in R$, then
$$ 2x^2=\begin{cases}
0  &  \textrm{ if } x=kb,k=1,\cdots,p-1  \\
mb \textrm{ (for some } m)  &  \textrm{ otherwise.}
\end{cases}
$$  \par
  If $ R=D $ and   $x\in R$, then $ 2x^2=0$.\par
Thus $ R $ satisfies Condition $ (ii) $ of Theorem
\ref{conter}. So $ S $  is an
artinian ring which is not \g right p.p.\\

 Since by Proposition~\ref{T(R)1}, $T_n(R)$ satisfies the assumptions of Theorem \ref{conter},  the group ring $T_n(R)G$ of
the group $G =\{e,g\}$ of order $ 2 $ over the  triangular matrix ring $ T_n(R) $, for each $n\geq 2$,  is also an
artinian ring which is not \g right p.p.
\end{example}

Let $R$ be a ring. Consider the subring  $T(R, n)$ of the triangular matrix ring  $T_n(R)$, with $n\geq 2$, consisting  all $n$ by $n$ triangular matrices with constant diagonals. We can denote elements of $T(R, n)$ by
$(a_1, a_2, \ldots , a_n)$, then $T(R, n)$ is a ring with addition pointwise and multiplication
given by  $$(a_1, a_2, \ldots , a_n)(b_1, b_2, \ldots , b_n)=  (a_1b_1,
a_1b_2+a_2b_1, \ldots
,a_1b_{n}+a_2b_{n-1}+\cdots+a_{n}b_1),$$ for each $a_i, b_j\in
R$.

On the other hand, there is a ring isomorphism $\varphi : R[x]/(x^n) \rightarrow T (R, n)$, given by $$\varphi(
 a_1+a_2x +\cdots+a_nx^{n-1}+(x^n)) = (a_1, a_2,\ldots , a_{n}),$$ with
$a_i \in R$ and $ 1 \leq i \leq n $. So $T (R, n) \cong
R[x]/(x^n)$, where $R[x]$ is the ring of polynomials in an indeterminant $ x $  and $(x^n)$ is the ideal generated by $x^n$.

\begin{proposition}\label{triangular}
Let $ R $ be an abelian ring.  Then $  R$ is   generalized right p.p.  if and only if $T(R, n)$ is   generalized right p.p.
\end{proposition}
\begin{proof}
The proof is similar to that of \cite[Proposition 3]{11}.
\end{proof}
\begin{proposition}\label{artinian}
A ring $  R$ is  right (left) artinian  if and only if the ring $T(R, n)$ is  right (left) artinian.
\end{proposition}
\begin{proof}
The proof is similar to that of \cite[Corollary 4.3]{nasr}.
\end{proof}
\begin{example}\label{ex50}
Let $ S $ be the group ring considered in Example \ref{ex}. Then by Example \ref{ex}, Propositions \ref{triangular} and  \ref{artinian},     the  ring   $\mathfrak{S}:=T(S, n)$
 is an artinian  ring which is  not  \g right p.p.
\end{example}

\begin{proposition}[\cite{ma3}, Theorem 3.14]\label{group}
Let $R$ be a ring and $G$ be a group. If the group ring $RG$ is generalized
right p.p., then so is $R$.
\end{proposition}

Repeatedly applying Proposition \ref{group} to existing examples, such as Example \ref{ex} or Example \ref{ex50}, one can  construct new  examples from old.
\begin{example}
Let $ S $ be the group ring considered in Example \ref{ex} or the ring $\mathfrak{S}$ in the Example  \ref{ex50} and $H$ be any finite group. Then by Proposition \ref{group}, Examples \ref{ex}, \ref{ex50} and \cite[~p.~217, Exercise 2]{intro},     the new group ring  $SH$  (respectively $\mathfrak{S}H$)
 is an artinian  ring that is also not  \g right p.p.
\end{example}

Let $ R $ be a ring and $ G $ be a group.  If  $ R $ has an involution $ \s $ itself then
  we have a natural involution $  \divideontimes$ on the group ring $ RG $, induced by the inversion in the
group $G$, given by
\begin{equation*}
(\sum a_gg)^{\divideontimes}=\sum a_g^{\ast} g^{-1}.
\end{equation*}
\indent Recall from \cite{ma3}, that a ring $ R $ with an involution $ \ast $ is called    a \textit{generalized~\rr~\s-ring}
  if for each  $x\in R$, the right annihilator
$r_R(x^n)$  is generated by a projection  for some positive integer $n$,  depending on $x$. These rings are   generalization of Rickart \s-rings. There are large classes of both finite and infinite dimensional Banach \s-algebras which are generalized
\rr \s-rings, but they are not Rickart $\ast$.
\begin{theorem}\label{conter2}
 Let $ R $ be a non-unital  right  artinian   ring   with the following properties:
\begin{itemize}
\item[$(i)$] For $ a\in R $ and each positive  integer $ n $, $ 3^na=0 $ implies that $ a=0 $;
\item[$(ii)$] The equation $ 3x^2+x=0 $ has only the trivial solution  in $ R $;
\end{itemize}
and let $S = RG$ be the group ring of the group $G =\{e,g,g^2\}$ of
order $ 3 $ over the ring $ R $. Then the group ring $ S $ is a right artinian
ring which is not  \g Rickart $
\ast $.
\end{theorem}
 \begin{proof}
By \cite[p.217, Exercise 2]{intro}, $ S $ is right artinian. But, we show that $ S $ is not a \g
\rr\s-ring.  Note that for each positive integer $ n
$, $ (e+g+g^2)^n=3^{n-1}e+3^{n-1}g+3^{n-1}g^2 $, $ (e-g)\in r_S(e+g)^n $, so $ r_S(e+g+g^2)^n\not=0 $. If $ ae+bg+cg^2\in S $ is a nontrivial projection such
that $ ae+bg+cg^2\in r_S(e+g+g^2)^n $, then we must have
\begin{align*}
&(3^{n-1}e+3^{n-1}g+3^{n-1}g^2 )( ae+bg+cg^2)=0;\\
&( ae+bg+cg^2)( ae+bg+cg^2)=( ae+bg+cg^2);\\
&( ae+bg+cg^2)=( ae+bg+cg^2)^{\s}=( ae+cg+bg^2).
\end{align*}
So $ 3^{n-1}a+3^{n-1}b+3^{n-1}c=0
$, $ a^2+bc+cb=a $, $ c^2+ab+ba=b $, $ b^2+ac+ca=c $ and $ b=c $. Then $ a=-2b$ and  $  b^2+ab+ba=b $. Hence  $ 3b^2+b=0
,$ which contradicts the assumption that the equation $
3x^2+x=0 $ has no nontrivial solutions in $ R $. Thus  $ r_S(e+g+g^2)^n
$ can not be generated by a projection of $ S$.
 \end{proof}
 \begin{proposition}\label{T(R)2}
 If a ring $R$ satisfies the assumptions of Theorem \ref{conter2}, then for each $n$, the  ring $T_n(R)$   also satisfies    these properties.
 \end{proposition}
 \begin{proof}
 The proof is similar to that of Proposition \ref{T(R)1}.
 \end{proof}
The following examples also satisfy the assumptions of Theorem \ref{conter2}.
\begin{example}\label{ex2}
 Let $ R $ be  one of the following finite rings of order $ p^2 $:
\begin{itemize}
\item[$(1)$] $A=\langle a \mid p^2a=0,a^2=pa\rangle$;
\item[$(2)$]$B=\langle a \mid p^2a=0,a^2=0\rangle$;
\item[$(3)$]$C=\langle a,b \mid pa=pb=0,a^2=b, ab=0\rangle$;
\item[$(4)$] $D=\langle a,b \mid pa=pb=0,a^2=b^2=0\rangle$.
\end{itemize}
where $p\not= 3$ is a prime number.  Now, let $S = T_n(R)G$ be the
group ring of the group $G =\{e,g,g^2\}$ of order $ 3 $ over the ring $
T_n(R) $. Since $ S $ is a finite ring, $ S$ is artinian. Also, since characteristic of $ R $ is $ p^2 $ and $ (p^2,3^n)=1 $,
for each $ n $, $ R $ satisfies Condition $ (i) $ of Theorem
\ref{conter2}. \par
If $ R=A $ and   $x\in R$, then
$$
 3x^2=
\begin{cases}
0  &  \textrm{ if } x=kp, \textrm{ where }k=1,\cdots,p-1\\
mb \textrm{ (for some } m)  &  \textrm{ otherwise.}
\end{cases}$$
\par
\indent  If $ R=B $ and   $x\in R$, then $ 3x^2=0$. \par
  If $ R=C $ and   $x\in R$, then
$$ 3x^2=\begin{cases}
0  &  \textrm{ if } x=kb,k=1,\cdots,p-1  \\
mb \textrm{ (for some } m)  &  \textrm{ otherwise.}
\end{cases}
$$ \par
   If $ R=D $ and   $x\in R$, then $ 3x^2=0$.\par
   
   Thus $ R $ satisfies Condition $ (ii) $ of Theorem
\ref{conter2}. So $ S $  is an
artinian ring which is not  \g Rickart $ \ast $.\\
\newpage
We define an involution on $T_n(R)$  given by   
 $ A^{\ast}=(a_{\ell k}^{\ast}) $, where $ \ell=n-j+1 $ and $ k=n-i+1 $, for $ A=(a_{ij})\in T_n(R) $ (see \cite{ma3}).
 Since by Proposition~\ref{T(R)2}, $T_n(R)$ satisfies the assumptions of Theorem \ref{conter2},  the group ring $T_n(R)G$ of
the group $G =\{e,g,g^2\}$ of order $ 3 $ over the  triangular matrix ring $ T_n(R) $, for each $n\geq 2$,  is also an
artinian ring which is not  \g Rickart-$ \ast $.
\end{example}

\bibliographystyle{amsplain}

\end{document}